\documentclass[]{article}

\usepackage{graphicx} 
\usepackage[a4paper,left=35mm,right=35mm,top=30mm,bottom=30mm,marginpar=25mm]{geometry}
\usepackage{amsmath}
\usepackage{amsthm}
\usepackage{amssymb}
\usepackage[monochrome]{xcolor}
 \usepackage[pdfpagelabels=false,colorlinks=true,linkcolor=blue,citecolor=red,urlcolor=cyan]{hyperref}
\usepackage{subcaption}
\usepackage[displaymath,mathlines]{lineno}

\allowdisplaybreaks

\newcommand\sizefigure{0.45} \newcommand\szfigrhos{0.45} \renewcommand{\div}{\operatorname{div}}

\newcommand{\Rr}{{\mathbb{R}}}

\numberwithin{theorem}{section}
\newtheorem{proposition}{Proposition}

\newtheorem{lemma}{Lemma}

\newtheorem*{remark}{Remark}
\makeindex

\author{Diogo A. Gomes and Roberto M. Velho        \thanks{This work was partially supported by KAUST baseline and start-up
                funds and KAUST SRI, Uncertainty Quantification Center in Computational Science
                and Engineering.}        \thanks{King Abdullah
                University of Science and Technology (KAUST), CEMSE Division and  
                KAUST SRI, Uncertainty Quantification Center in Computational Science
                and Engineering, Thuwal 23955-6900, Saudi Arabia.                 {\tt\small diogo.gomes@kaust.edu.sa}
                {\tt\small roberto.velho@gmail.com}
        }}

\title{\Large \bf On the Hughes Model and Numerical Aspects}

\begin{document}

\maketitle
\date{}
\begin{abstract}
Here, we study a crowd model proposed by R. Hughes in \cite{hug02} and we describe a numerical approach to solve it. The Hughes model comprises a Fokker-Planck equation coupled with an eikonal equation with Dirichlet or Neumann data.
First, we establish a priori estimates for the solutions. Second, we study  radial solutions and identify a shock
formation mechanism. Third, we illustrate the existence of congestion, the breakdown of the model, and 
the trend to the equilibrium. Finally, we propose a new numerical method and  
consider two numerical examples.
\end{abstract}

\section{Introduction}

Understanding the dynamics of pedestrian crowds is of great  significance, in particular for the prevention of catastrophic emergency
evacuations. Here, we consider an extension of the PDE\ model  proposed 
in \cite{hug02} that describes the evolution of a pedestrian crowd. 
Our  system of PDEs comprises a continuity equation or Fokker-Planck equation with viscosity~ $\varepsilon \geq  0$ and a Hamilton-Jacobi equation
\begin{equation}\label{eq_hughes_fokker-planck}
\begin{cases}
\displaystyle \rho_t(x,t) - \div(\rho(1-\rho)^2 Du)=\varepsilon  \Delta \rho,\\
\displaystyle |Du(x)|^2=\frac{1}{(1-\rho)^2},  
\end{cases}
\end{equation} 
where $\rho: \Omega \times \Rr_+ \to \Rr$,  the density of agents,  and $u: \Omega \to \Rr$,  the exit time, are the unknowns and  the given initial data is $\rho(x,0) = \rho_0(x)$ and $u(x)=u_0$, with $x$ representing the spatial variable and $t$ the time. Here, $\Omega$ is an open domain of $\Rr^d$ and we focus on the cases relevant in applications,  $\Rr$ and $\Rr^2$.
The Fokker-Planck equation describes the evolution of the crowd density $\rho$
($0 \leq \rho \leq 1$), while the Eikonal determines the optimal direction
of movement for each individual/agent if they assume that the rest of the
population is frozen. The constraint $\rho\leq 1$ corresponds to the
maximal density of the population.
The Dirichlet condition corresponds to  the areas on the boundary where agents/people/pedestrians can leave. The Neumann condition corresponds to a no-flow condition at the boundary. The correct interpretation of these boundary conditions is essential  in the design of  numerical schemes. The case $\varepsilon=0$\ is the model introduced in \cite{hug02}. 

Significant progress has been achieved in the understanding of these problems (\cite{MR2921876}, \cite{markw}, \cite{MR3229831}, \cite{MR3041567}). However, even some aspects of one-dimensional models are not completely understood.
Microscopic \cite{MR3460619, 2016arXiv160206153D} and mean-field game \cite{markw2, MR2928383} interpretations were used to study  the macroscopic dynamics that the Hughes model describes.
Numerical approaches to these problems were developed in \cite{2016arXiv160107324C, MR3055243}.

Here, we establish new a priori estimates for solutions of  \eqref{eq_hughes_fokker-planck} that give a partial regularity for the solutions. Then, we consider a radial problem to examine the behavior of the model in two and three dimensions in a simplified setting. We show that, when $\varepsilon=0,$ the model admits shocks, which  we also illustrate numerically. Next, we study a one-dimensional problem, the flow problem. Here,  agents arrive at a prescribed rate, that we call current,  on one side of an interval and leave through the other side. We show that, in certain circumstances, the Hughes model may not be well-posed as the density $\rho$ exceeds the maximal congestion threshold. In these examples, we identify two  mechanisms of loss of regularity:  shocks in the zero-viscosity problem and congestion in the flow problem.  Finally, we discuss a new numerical method for
\eqref{eq_hughes_fokker-planck} and illustrate the trend to equilibrium and two numerical examples. 

\section{Estimates}
As a first attempt to understand the existence of solutions to the Hughes model, we investigate a priori estimates; that is, estimates that are valid
for smooth enough solutions. These estimates extend some of the results in
 \cite{MR3195844} for  the periodic setting. First, for the Neumann or Dirichlet boundary conditions and $\varepsilon>0$, we prove that smooth solutions satisfy $0\leq \rho \leq 1$ for all times, if the initial condition also satisfies this condition. This should be contrasted with the flow problem considered in Section \ref{fpsec}. Next, we prove that $Du\in L^p$, for any $1<p<\infty$, and then $\frac{1}{1-\rho}\in L^p$. This last estimate gives a quantitative control on the congestion. 
Here, in contrast, we consider the Hughes model in an open domain, $\Omega\subset \Rr^d$, with Dirichlet-Neumann conditions: $\partial \Omega=\Gamma_d\cup\Gamma_n$,  $\nu$ the outer unit normal to $\partial\Omega$, $\rho=0$ in $\Gamma_d\times [0,T]$, $\rho_\nu=0$ in $\Gamma_n\times [0,T]$ and viscosity $\varepsilon=1$.\\

\begin{lemma}\label{rbs}
Let $\rho:\Omega\times [0,T]\to \Rr$ solve
\[
\rho_t - \div(\rho (1-\rho) g(x,t))=\Delta \rho.  
\]

Then, $0\leq \rho(x,t)\leq 1$, if $0\leq \rho(x,0)\leq 1$.\\
\end{lemma}

\begin{proof}
Note that $\tilde \rho=1-\rho$ satisfies
\begin{displaymath}
\tilde \rho_t-\div(\rho \tilde \rho g(x,t))=\Delta \tilde \rho.
\end{displaymath}
Because $\tilde \rho(x,0)\geq 0$ (and, with Dirichlet boundary data,   $\tilde \rho \geq 0$ in $\partial \Omega$),  
we have $\tilde \rho \geq 0$. 
\end{proof}

\begin{proposition}\label{propIdn}
        Let $(u, \rho)$ solve \eqref{eq_hughes_fokker-planck} with $\varepsilon = 1$. Suppose
        $u=0$ in $\Gamma_d\times [0,T]$, $u_\nu=0$ on  $\Gamma_n\times [0,T]$                         and
        $0< \rho< 1$ at $t=0$.
        Then, for any $\alpha<-1$
        \begin{displaymath}
        \displaystyle   \frac{d}{dt}\int_{\Omega} (1-\rho)^{\alpha+1}\leq C \int_{\Omega} (1-\rho)^{\alpha+1}.
        \end{displaymath}
        Furthermore, 
        \[
        \int_0^T \int_{\Omega}  |D(1-\rho)^{\frac{\alpha+1}{2}}|^2\leq C.
        \]
\end{proposition}

\begin{proof}
Multiply the first equation in \eqref{eq_hughes_fokker-planck} by $-(\alpha+1)(1-\rho)^\alpha$. Then,
\begin{align*}
 \frac{d}{dt}\int_{\Omega}  (1-\rho)^{\alpha+1} \leq& 
                      \ c \int_{\Omega}  (1-\rho)^{\alpha+1} \rho D\rho Du  
  -\alpha(\alpha+1) \int_{\Omega}  (1-\rho)^{\alpha-1}|D\rho|^2\\
& -(\alpha+1) \int_{\partial \Omega} (1-\rho)^{\alpha+2} \rho u_\nu  
 -(\alpha+1) \int_{\partial \Omega} (1-\rho)^{\alpha}\rho_\nu
\\ 
&{\color{blue}\text{(using Cauchy's inequality)}} \\
  &\leq -\frac{\alpha (\alpha+1)}2 \int_{\Omega}  (1-\rho)^{\alpha-1} |D\rho|^2 
  +\int_{\Omega} (1-\rho)^{\alpha+3} \rho^2 |Du|^2  \\
&-(\alpha+1) \int_{\partial \Omega} (1-\rho)^{\alpha+2} \rho u_\nu 
 -(\alpha+1) \int_{\partial \Omega} (1-\rho)^{\alpha}\rho_\nu\\ 
&\hspace{0cm} {\color{blue}\text{(using the Eikonal equation and the Lemma \ref{rbs})}} \\
    &\leq -\frac{\alpha (\alpha+1)}2 \int_{\Omega}  (1-\rho)^{\alpha-1} |D\rho|^2 
    +C \int_{\Omega}  (1-\rho)^{\alpha+1}\\
    & -(\alpha+1) \int_{\partial \Omega} (1-\rho)^{\alpha+2} \rho u_\nu 
                -(\alpha+1) \int_{\partial \Omega} (1-\rho)^{\alpha}\rho_\nu.
\end{align*}

Now, we observe that, on $\Gamma_n,$ $u_\nu=0$ and, on $\Gamma_d$, we have $\rho=0$. Hence,
\(
(1-\rho)^{\alpha+2} \rho u_\nu
= 0
\)
in $\partial \Omega$. Similarly,
\(
(1-\rho)^{\alpha}\rho_\nu\leq 0
\)
in $\Gamma_d$ and vanishes in $\Gamma_n$. Hence, is also non-positive in $\partial \Omega$.
Thus, taking into account that $\alpha+1\leq 0$,
integrating in time, and using Gronwall's inequality, we get the desired estimates. 
\end{proof}

\begin{proposition}
        \label{prointdudn}
        Under the same hypothesis of Proposition~\ref{propIdn}, we have that, for any $1~<~p~<~\infty$,  

        \[
        \displaystyle   \sup_{0\leq t \leq T}  \int_{\Omega} |Du|^{2p}<C_p.
        \]
\end{proposition}
\begin{proof}
We use the first conclusion of Proposition~\ref{propIdn} in the Eikonal equation and observe that
\begin{displaymath}
|Du|^2 = \frac{1}{(1-\rho)^2} \in L^p, \ \forall p.
\end{displaymath}
\end{proof}

\section{Shocks in radial solutions}
To understand the behavior of the Hughes model, we consider radial solutions. Thus, equation \eqref{eq_hughes_fokker-planck} becomes a scalar PDE and, thanks to this simplification, we identify the formation of shocks in the zero viscosity problem. We expect shocks to exist in general two and three-dimensional problems. 

Now, we assume radial symmetry  corresponding to a model where agents want to get away from the origin. In dimension $d > 1$, assume
 $u=u(r,t)$, where $r$ is the radius. The eikonal equation in \eqref{eq_hughes_fokker-planck} gives that
\begin{equation*}
u_r = \pm \ \frac{1}{1-\rho}.
\end{equation*}
We select the negative root because it corresponds to agents leaving the origin. Because $\rho$ is radial,  $\rho = \rho(r,t)$, we rewrite the Fokker-Planck equation in \eqref{eq_hughes_fokker-planck} in polar coordinates.  Using the preceding equation, we get
\begin{equation}\label{eq_radial_rho}
\rho_t + \frac{d-1}{r} \rho (1-\rho) + \rho_r (1- 2 \rho) = \varepsilon \left[\rho_{rr} + \frac{d-1}{r} \rho_r \right].
\end{equation} 
When $\varepsilon=0$,  the previous equation becomes the first-order partial differential equation
\begin{equation}\label{eq_radial_rho_epsilon_zero}
\rho_t + \frac{d-1}{r} \rho (1-\rho) + \rho_r (1- 2 \rho) = 0.
\end{equation}
As it is usual for first-order nonlinear partial differential equations, shocks can arise. We study the shocks using the method of characteristics.

First, we solve the characteristic system\begin{equation}\label{eq_characteristics_r_neq_inf}
\begin{cases}
\displaystyle \frac{d r}{d t} = 1 - 2 \rho,\vspace{0.1cm}\\
\displaystyle \frac{d \rho}{d t} = - \frac{d-1}{r} \rho (1-\rho),\\
\end{cases}
\end{equation}
with $r_0$, $\rho_0(r_0)$ as initial conditions in time zero for $r(t)$ and $\rho(t)$, respectively.

To solve this system of ODEs, we first define $V$ as 
\begin{equation} \label{eq_definition_V}
V(\rho) = \rho (1-\rho),
\end{equation}
and, analogously, 
$V_0(\rho_0) = \rho_0 (1 - \rho_0)$.  Next,  we rewrite \eqref{eq_characteristics_r_neq_inf} as
\begin{equation}\label{eq_charac_2nd_form}
\begin{cases}
\displaystyle \frac{d r}{d t} = \frac{\partial V}{\partial \rho},\\
\displaystyle \frac{d \rho}{d t} = - \frac{d-1}{r} V.\\
\end{cases}
\end{equation}
Then, we compute $$ \frac{\partial V}{\partial r} = - \frac{d-1}{r} \ V,$$ and conclude that\begin{equation}\label{eq_sol_V}
V = \left( \frac{r_0}{r} \right)^{d-1}  V_0.
\end{equation}

Now, solving for $\rho$ in \eqref{eq_definition_V} gives two roots. If $0<\rho_0(r_0)\leq1/2$, here called {\em regime 1},
\begin{equation*}
\rho_1 = \frac{1 -  \sqrt{1 - 4V }}{2},
\end{equation*}
while, if $1/2<\rho_0(r_0) < 1$, here called {\em regime 2},
\begin{equation*}
\rho_2 = \frac{1 +  \sqrt{1 - 4V }}{2}.
\end{equation*}
Using these expressions on the R.H.S. of the ODE for $r(t)$ in \eqref{eq_characteristics_r_neq_inf} and using \eqref{eq_sol_V}, we get

\begin{equation}\label{eq_r1dot_d_dimensional}
\dot{r_1}(t) = 1 - 2\rho_1 =
 \sqrt{1 - 4 V_0 \left[\frac{r_0}{r_1(t)}\right]^{d-1}},
\end{equation}
and
\begin{equation}\label{eq_r2dot_d_dimensional}
\dot{r_2}(t) = 1 - 2\rho_2 =
 - \sqrt{1 - 4 V_0 \left[\frac{r_0}{r_2(t)}\right]^{d-1}}.
\end{equation}

Now, using \eqref{eq_charac_2nd_form} and \eqref{eq_sol_V}, we obtain the ODE describing the time evolution of $\rho$:
\begin{equation}\label{eq_ODE_rho_d_dimensional}
\dot{\rho}(t) = - \frac{d-1}{r(t)} \left[ \frac{r_0}{r(t)} \right]^{d-1} V_0. 
\end{equation}

Along the next two subsections, we present the particular cases of the radial solutions in dimension $2$ and $3$. We also make use of the following remark.

\begin{remark}
        The function defined by $x \mapsto \sqrt{1 - 4 x (1-x)}$, is identical to $ -2x +1 $ in the interval $[0,1/2]$ and identical to $ 2x -1$ in the interval $[1/2,1]$.
\end{remark}

\subsection{Dimension 2}
For the Hughes model in dimension $2$ we must solve the following equations:
\begin{equation*}\label{eq_r1dot_d_dimensional}
\dot{r_1}(t) = \sqrt{1 - 4 V_0 \frac{r_0}{r_1(t)}},\ \text{with } 0 < \rho_0(r_0) \leq 1/2,
\end{equation*}
and
\begin{equation*}\label{eq_r2dot_d_dimensional}
\dot{r_2}(t) = - \sqrt{1 - 4 V_0 \frac{r_0}{r_2(t)}},\ \text{with } 1/2 < \rho_0(r_0) < 1,
\end{equation*}
with the initial conditions $r_1(0)=r_2(0)=r_0$.

We could only obtain a solution to $r_1(t)$ and $r_2(t)$ in implicit forms.

{\bf Regime 1:}
$r_1(t)$ is expressed as:
\begin{align*}
2 r_0 \rho_0 (1-\rho_0) \log \left[ 2 r_1(t) \left( 1 + \sqrt{1 - \frac{4 r_0 \rho_0 (1-\rho_0)}{r_1(t)}}   \right)  - 4 r_0 \rho_0 (1-\rho_0) \right] +\\
+ r_1(t) \sqrt{1 - \frac{4 r_0 \rho_0 (1-\rho_0)}{r_1(t)}} =
 t + r_0 \left\{ 2 \rho_0 (1-\rho_0) Log \left[ 4 r_0 (1-\rho_0)^2  \right] - 2 \rho_0 + 1 \right\}.
\end{align*}

{\bf Regime 2:}
$r_2(t)$ is expressed as:

\begin{align*}
2 r_0 \rho_0 (1-\rho_0) Log \left[ 2 r_2(t) \left( 1 + \sqrt{1 - \frac{4 r_0 \rho_0 (1-\rho_0)}{r_2(t)}}   \right)  - 4 r_0 \rho_0 (1-\rho_0) \right] +\\
+ r_2(t) \sqrt{1 - \frac{4 r_0 \rho_0 (1-\rho_0)}{r_2(t)}} =
-t + r_0 \left\{ 2 \rho_0 (1-\rho_0) Log \left[ 4 r_0 \rho^2_0 \right] + 2 \rho_0 - 1 \right\},
\end{align*}

Now, in dimension $2$, equation \eqref{eq_ODE_rho_d_dimensional} for the time evolution of the density $\rho$  becomes

\begin{displaymath}
\dot{\rho}(t) = - \frac{r_0}{{r(t)}^2}V_0. 
\end{displaymath}

We solve it and present the parametric plot of the radius $r(t)$ versus the density $\rho(t)$ in subsection \ref*{sct_numerical_plot_radial_solutions}. 

\subsection{Dimension 3}
For the Hughes model in dimension $3,$ we must solve the following equations:

\begin{equation*} \dot{r_1}(t) = \sqrt{1 - \frac{4 V_0 {r_0}^2}{{r_1(t)}^2}},\ \text{with } 0 < \rho_0(r_0) \leq 1/2,
\end{equation*}
and
\begin{equation*} \dot{r_2}(t) = - \sqrt{1 - \frac{4 V_0 {r_0}^2}{{r_2(t)}^2}},\ \text{with } 1/2 < \rho_0(r_0) < 1,
\end{equation*}
with the initial conditions $r_1(0)=r_2(0)=r_0$.

We obtain the explicit formulas
\begin{align*}
r_1(t) = \sqrt{4 {r_0}^2 V_0 + (t+C_1)^2} 
= \sqrt{4 {r_0}^2 V_0 + t^2 + 2 t C_1 + C_1^2},
\end{align*}
and
\begin{align*}
r_2(t)  =\sqrt{4 {r_0}^2 V_0 + (t-C_2)^2} 
        =\sqrt{4 {r_0}^2 V_0 + t^2 - 2 t C_2 + C_2^2} ,
\end{align*}
where $C_1$ and $C_2$ are chosen so that $r(0)=r_0$, implying $C_1=C_2=\pm r_0\sqrt{1 - 4 V_0}$.\\

The remark implies that $C_1=r_0 (-2 \rho_0 +1)$ and $C_2=r_0 (2 \rho_0 -1)$. Plugging them back in the expressions for $r_1(t),\ r_2(t)$ we obtain:
\begin{equation*}
r_1(t) = \sqrt{t^2 + 2 r_0 (1-2\rho_0) t + {r_0}^2}, \end{equation*}
and
\begin{equation*}
r_2(t) = \sqrt{t^2 - 2 r_0 (2 \rho_0 -1) t + {r_0}^2}. \end{equation*}
Thus, the expressions for 
  $r_1(t)$ and  $ r_2(t)$ agree and there is no need to consider two separate regimes. Thus, we get\begin{equation}\label{sol_r}
r(t) = \sqrt{t^2 + 2 r_0 (1-2\rho_0) t + {r_0}^2 }.
\end{equation}

Now, the ODE for  the density $\rho(t)$ \eqref{eq_ODE_rho_d_dimensional} assumes the form
\begin{displaymath}
\dot{\rho}(t) = - 2 \frac{{r_0}^2}{{r(t)}^3}V_0, 
\end{displaymath}
that solved with the initial condition $\rho(0)=\rho_0$ has the solution

\begin{align}
\rho(t)  = \frac{1}{2} \left[ 1 - \frac{t + r_0 (1-2\rho_0)}{r(t)} \right] 
        = \frac{1}{2} \left[ 1 - \frac{t + r_0 (1-2\rho_0)}{\sqrt{t^2 + 2 r_0 (1-2\rho_0) t + {r_0}^2}} \right]. 
\end{align}

\subsection{Numerical experiments}\label{sct_numerical_plot_radial_solutions}

Now, we  numerically investigate the  formation of shocks for the radial Hughes model without viscosity. We construct three profiles for the density $\rho_0$, all of them with support (corresponding to the values of $r_0$) in $[0,1]$.

Next, we plot the graph corresponding to $r(t)$ versus $\rho(t)$ for each profile. The different colors correspond to the solution at the different times. The blue curve in each plot corresponds to the initial density profile $\rho_0$, see {\color{red}Figure}~\ref{fig:shocks_different_initial_profiles_radial_equation_2D} for the two-dimensional case and  {\color{red}Figure}~\ref{fig:shocks_different_initial_profiles_radial_equation} for the three-dimensional one. 

We observe the formation of shocks in the solutions via the graph of $r(t)$ versus $\rho(t)$. After a particular time, it becomes a non-single-valued function. In fact, once a shock happens, the characteristic's method is not valid and, consequently, the above expressions for $r(t)$ and $\rho(t)$ lack meaning.

\begin{figure}
\centering
\begin{subfigure}[b]{\szfigrhos\textwidth}
\includegraphics[width=\textwidth]{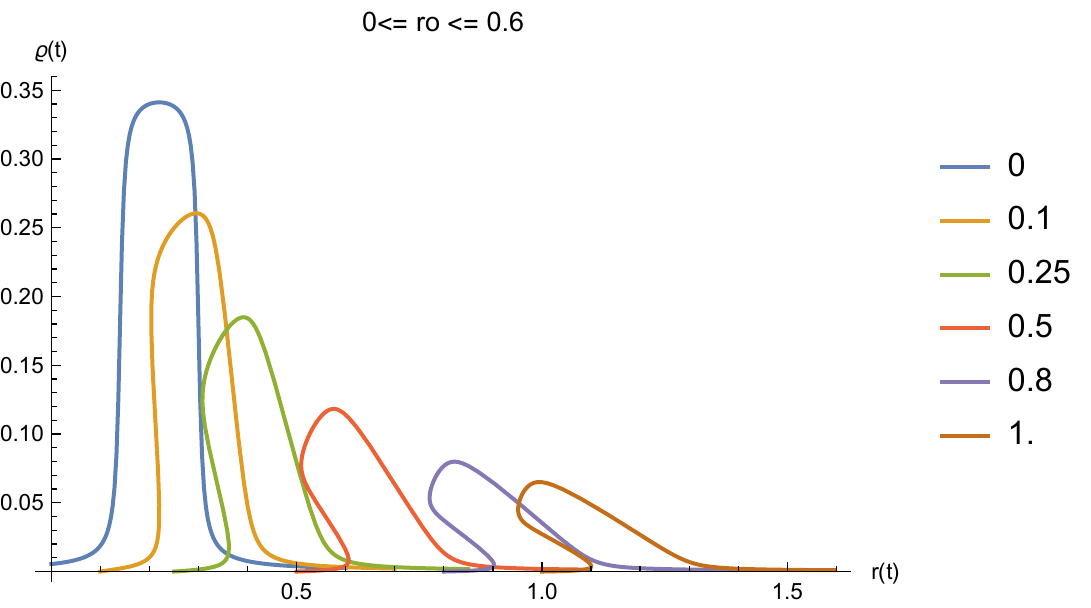}
\caption{\color{blue}Case 1: $\rho_0$ concentrated around $0.2$ with support on $0~\leq~ r_0~\leq~0.6$ and maximum value of $0.35$.}
\label{fig:shocks_case_1}
\end{subfigure}\vspace{0.25cm}
~ \begin{subfigure}[b]{\szfigrhos\textwidth}
\includegraphics[width=\textwidth]{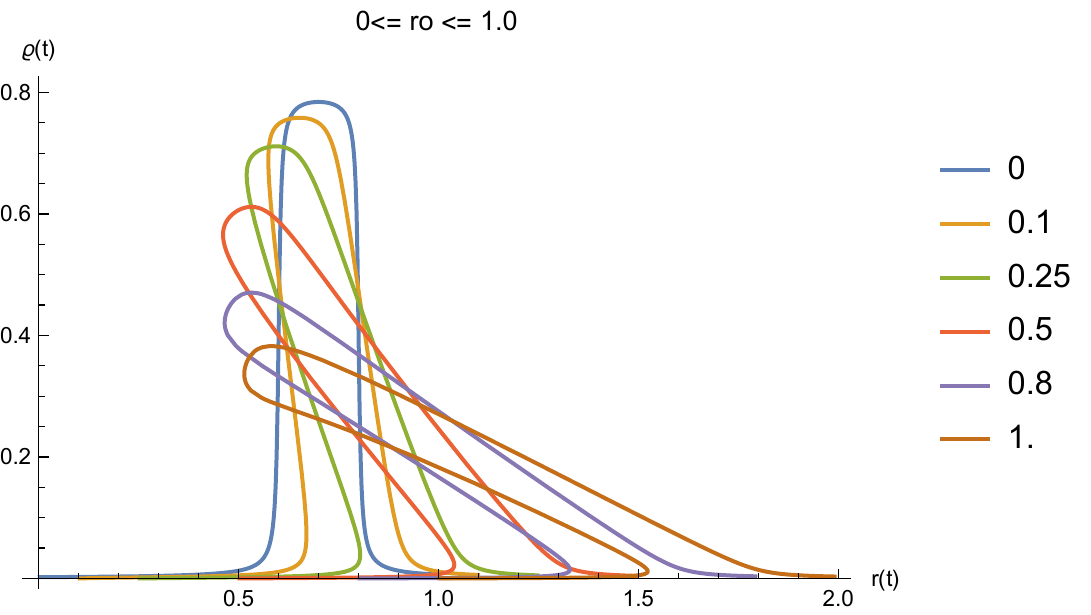}
\caption{\color{blue}Case 2: $\rho_0$ concentrated around $0.75$ with support on $0.5~\leq~ r_0~\leq~1.0$ and maximum value of $0.8$.}
\label{fig:shocks_case_2}
\end{subfigure}\vspace{0.25cm}
~
\begin{subfigure}[b]{\szfigrhos\textwidth}
\includegraphics[width=\textwidth]{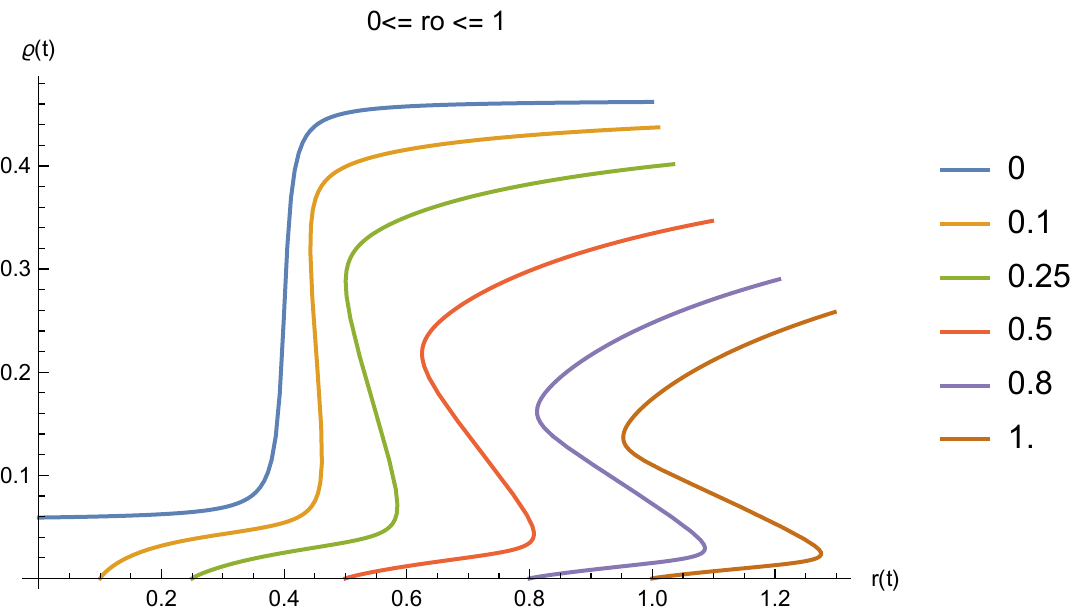}
\caption{\color{blue}Case 3: $\rho_0$ with support on $0~\leq~ r_0~\leq~1.0$ and maximum value of $0.4$.}
\label{fig:shocks_case_3}
\end{subfigure}
\caption{{\color{blue}Shocks along the time evolution of three given profiles in 2-D.}}
\label{fig:shocks_different_initial_profiles_radial_equation_2D}
\end{figure}

            \begin{figure}
            \centering
            \begin{subfigure}[b]{\szfigrhos\textwidth}
            \includegraphics[width=\textwidth]{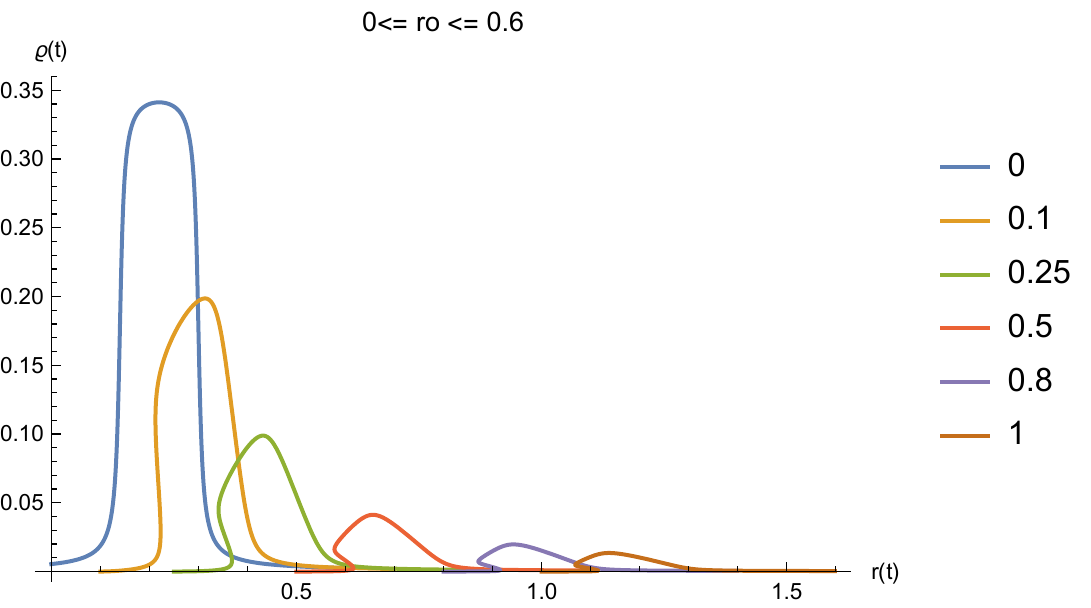}
            \caption{\color{blue}Case 1: $\rho_0$ concentrated around $0.2$ with support on $0~\leq~ r_0~\leq~0.6$ and maximum value of $0.35$.}
            \label{fig:shocks_case_1}
            \end{subfigure}\vspace{0.25cm}
            ~                         \begin{subfigure}[b]{\szfigrhos\textwidth}
            \includegraphics[width=\textwidth]{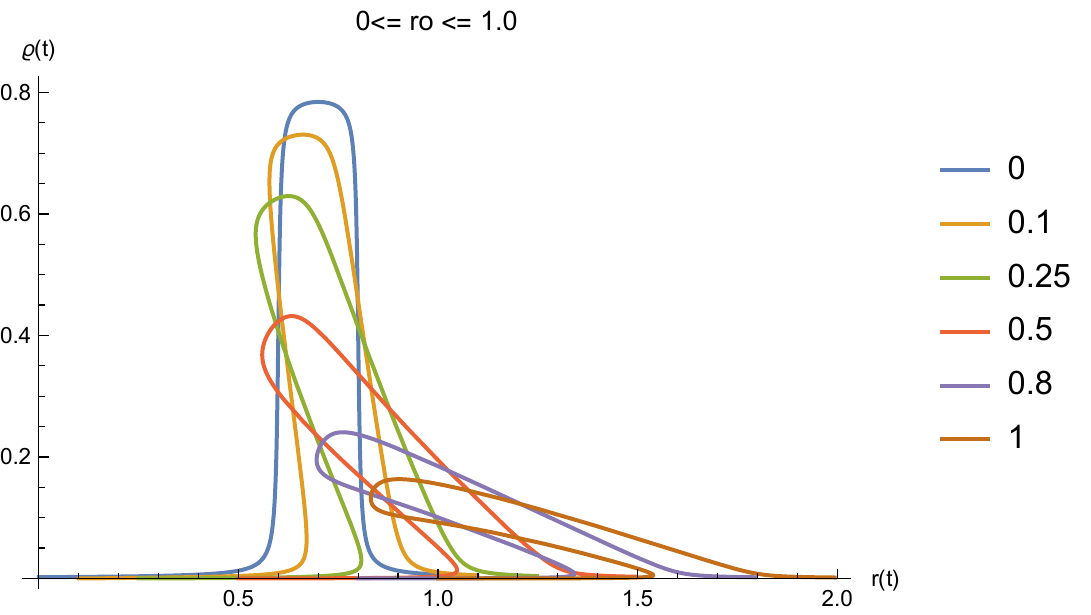}
            \caption{\color{blue}Case 2: $\rho_0$ concentrated around $0.75$ with support on $0.5~\leq~ r_0~\leq~1.0$ and maximum value of $0.8$.}
            \label{fig:shocks_case_2}
            \end{subfigure}\vspace{0.25cm}
            ~
            \begin{subfigure}[b]{\szfigrhos\textwidth}
            \includegraphics[width=\textwidth]{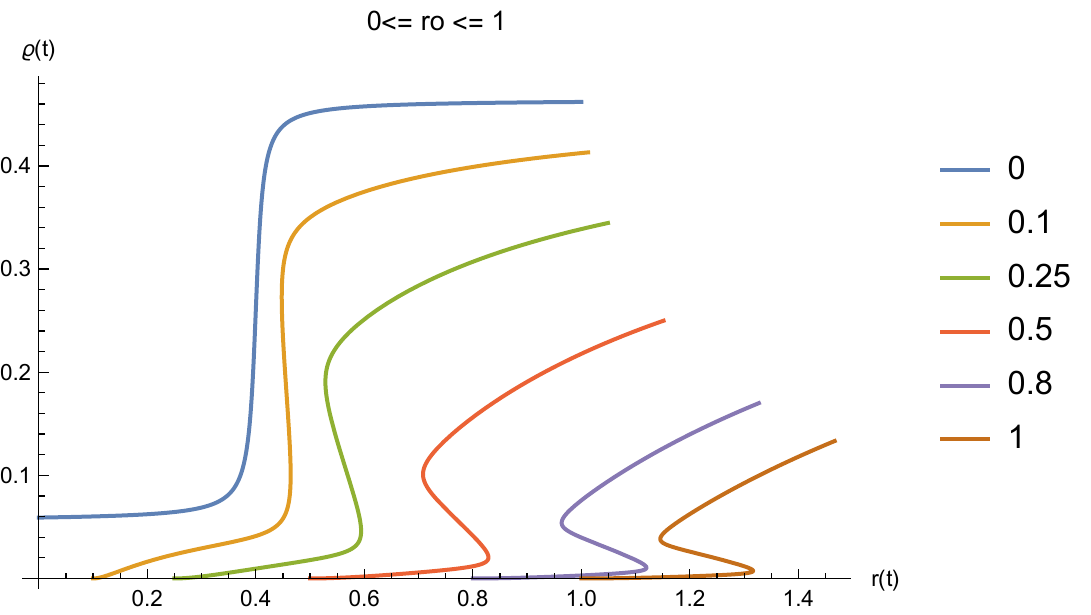}
            \caption{\color{blue}Case 3: $\rho_0$ with support on $0~\leq~ r_0~\leq~1.0$ and maximum value of $0.4$.}
            \label{fig:shocks_case_3}
            \end{subfigure}
            \caption{{\color{blue}Shocks along the time evolution of three given profiles in 3-D.}}
            \label{fig:shocks_different_initial_profiles_radial_equation}
            \end{figure}

Case $1$ and $2$ have initial profiles with supports of similar size but case $2$ presents shocks in a shorter time (both in 2-D and 3-D cases). This shows the dependence on the intensity of $\rho_0$; in case $2$, $\rho_0$ reaches $0.8$, while in case $1$, it is not greater than $0.35$.     

Case $1$ and $2$ exemplify the dependence on the value of $\rho_0$ for the time the shock appears. Both cases have initial profiles with supports of similar size but case $2$ presents shocks in a shorter time (both in 2-D and 3-D cases). In case 2, $\rho_0$ reaches $0.8$, while in case 1, it is not greater than $0.35$.

Finally, the presence of shocks in our examples gives the existence of shocks for the Hughes model in any dimension. We are not aware a proof of this feature in the literature.

\section{Flow problem - stationary case}\label{fpsec}

The flow problem is a natural problem in dimension one. It consists of people entering a domain from one side at a fixed rate and exiting through the other side. If the flow is large enough, the maximal density $\rho=1$ may be achieved as we show in what follows.  This situation illustrates the breakdown of the model.

Consider the one-dimensional flow problem in $[0,1]$ where agents arrive at $x=0$ and are only allowed to leave through $x=1$.
By computing $u_x$ in the Eikonal equation \eqref{eq_hughes_fokker-planck} and substituting in the Fokker-Planck equation \eqref{eq_hughes_fokker-planck}, we obtain
\begin{equation}\label{eq_pde_rho_1-dimensional}
\rho_t + \rho_x (1 - 2 \rho) = \varepsilon \ \rho_{xx}.
\end{equation}
The corresponding 
stationary equation is
\begin{equation*}
\varepsilon \rho_{xx} + 2\rho \rho_x - \rho_x = \frac{d}{dx} \left[\varepsilon \rho_x + \rho^2 - \rho \right] = 0.
\end{equation*}
We can then formulate the stationary flow problem as:
\begin{equation}\label{eq_ODE_stationary}
\begin{cases}
\varepsilon \rho_x + \rho^2 - \rho = j, \ \ \ x \in [0,1],\\
\rho(1) = 0,\\
\end{cases}
\end{equation}
where  $j$ is a prescribed net current of agents entering the domain.

Our interest is to understand the behavior of the solutions of \eqref{eq_ODE_stationary} as the current $j$ becomes large; that is, a large flow of agents. By solving numerically the ODE \eqref{eq_ODE_stationary},
 for $\varepsilon = 1$ and $j$ between $0$ and $1.5$, with increments of $0.1$, we observe the different solutions for~$\rho$. For $j > 1.2$, the density $\rho$ is larger than one and thus the model breaks down, see Figure \ref{fig:congestion}.

\begin{figure}
\centering
\begin{subfigure}[b]{\sizefigure\textwidth}
\includegraphics[width=\textwidth]{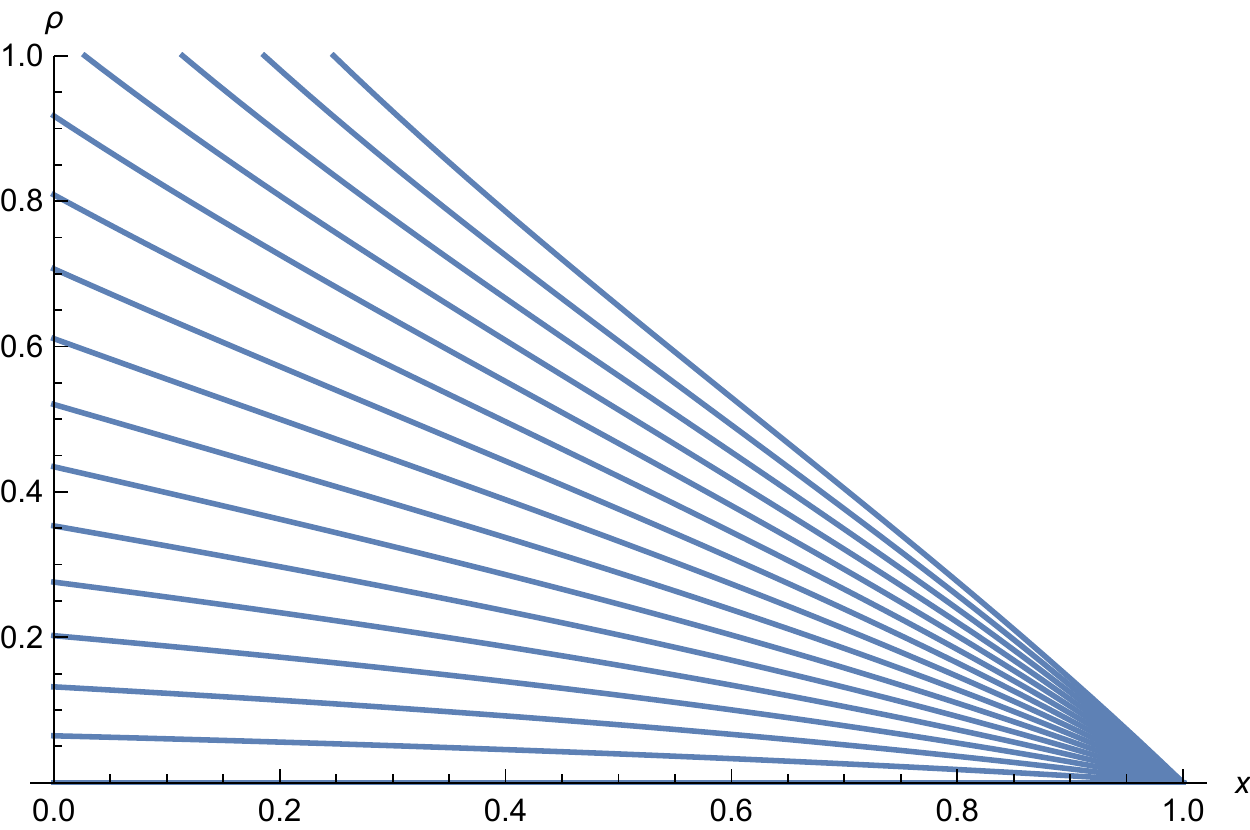}
\caption{{\color{blue}$\varepsilon = 1$} and $j$ from 0 to 1.5,}
\label{fig:congestion}
\end{subfigure}
~ \begin{subfigure}[b]{\sizefigure\textwidth}
\includegraphics[width=\textwidth]{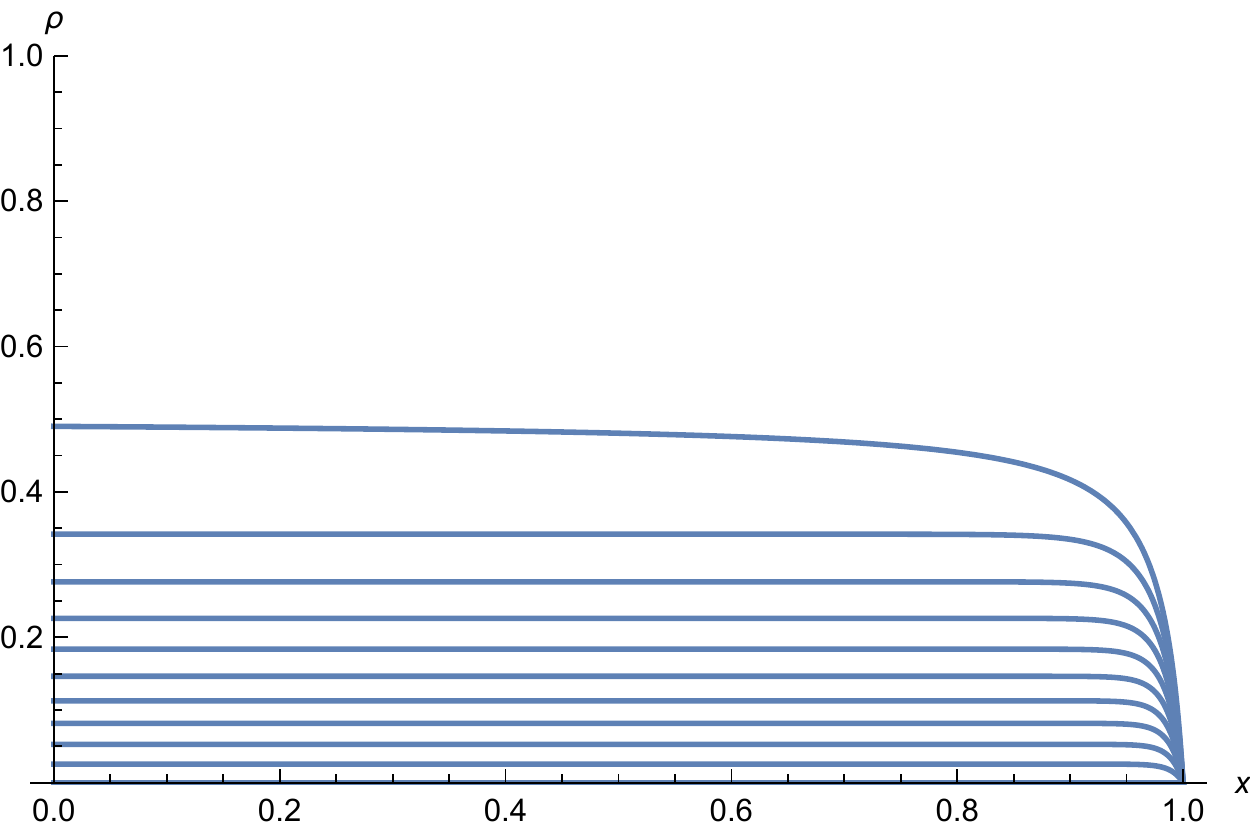}
\caption{{\color{blue}$\varepsilon = 0.01$} and $j$ from $0$ to $0.25$.}
\label{fig:congestion_low_viscosity}
\end{subfigure}
\caption{{\color{blue}Onset of congestion.}}
\label{figure_onset_congestion}
\end{figure}

This is remarkable because solutions of the reduced equation  \eqref{eq_ODE_stationary}
are $C^\infty$, however, in the original model the equations become singular. 

Now, in {\color{red}Figure \ref{fig:congestion_low_viscosity}}, we depict the effect of a small viscosity ($\varepsilon = 0.01$) for a range of admissible currents (the ones avoiding $\rho > 1$). We call the current for which the model stops to work of critical current.

\subsection*{Dependence on viscosity}
Here, we investigate the dependence of the viscosity on the solutions.
In {\color{red}Figure \ref{fig:supercritical_case}}, we see that, for small viscosity, the model breaks down. However, large viscosity seems to have a
stabilizing effect. 
In {\color{red}Figure \ref{fig:supercritical_case}}, we used a current with a fixed value {$j=0.5$} and viscosities from $0.3$ to $1.5$ with increments of
$0.1$.

\begin{figure}
\centering
\begin{subfigure}[b]{\sizefigure\textwidth}
\includegraphics[width=\textwidth]{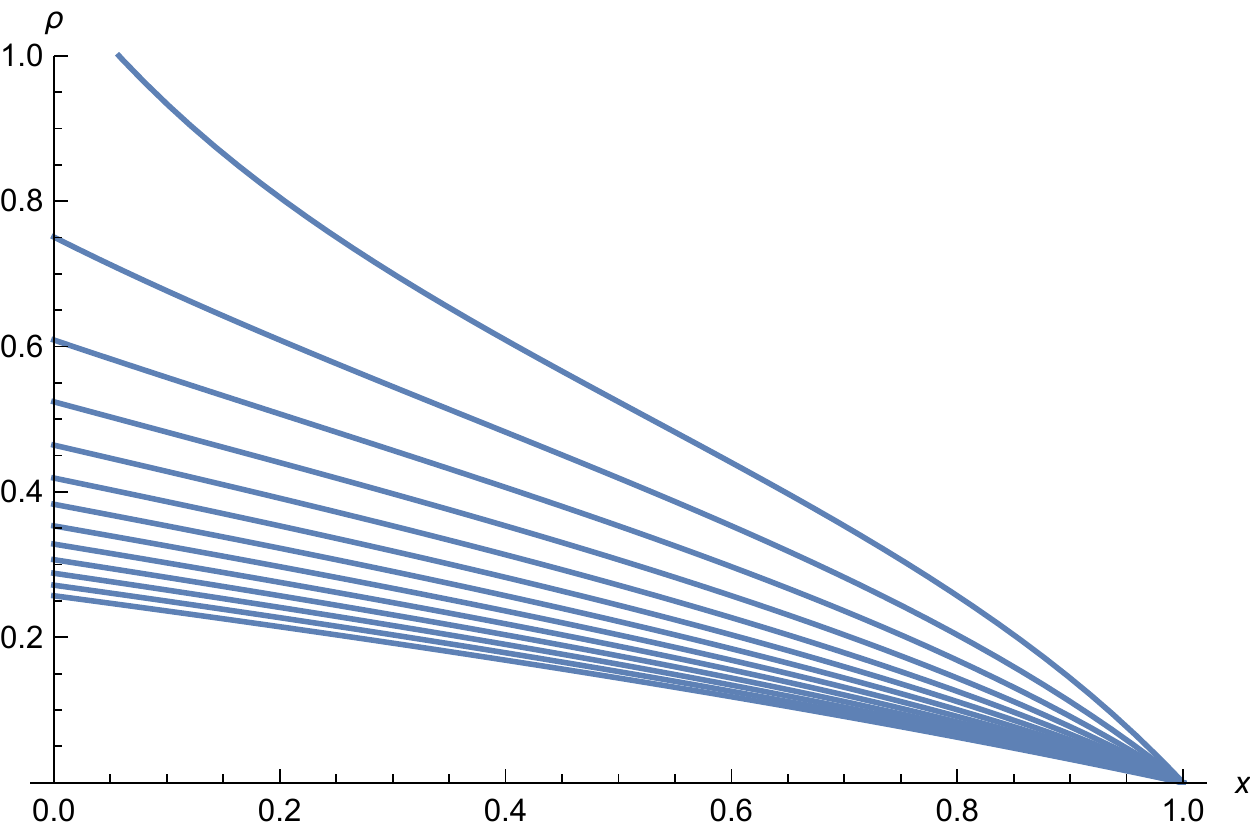}
\caption{\color{blue}Supercritical case,}
\label{fig:supercritical_case}
\end{subfigure}
~ 
\begin{subfigure}[b]{\sizefigure\textwidth}
\includegraphics[width=\textwidth]{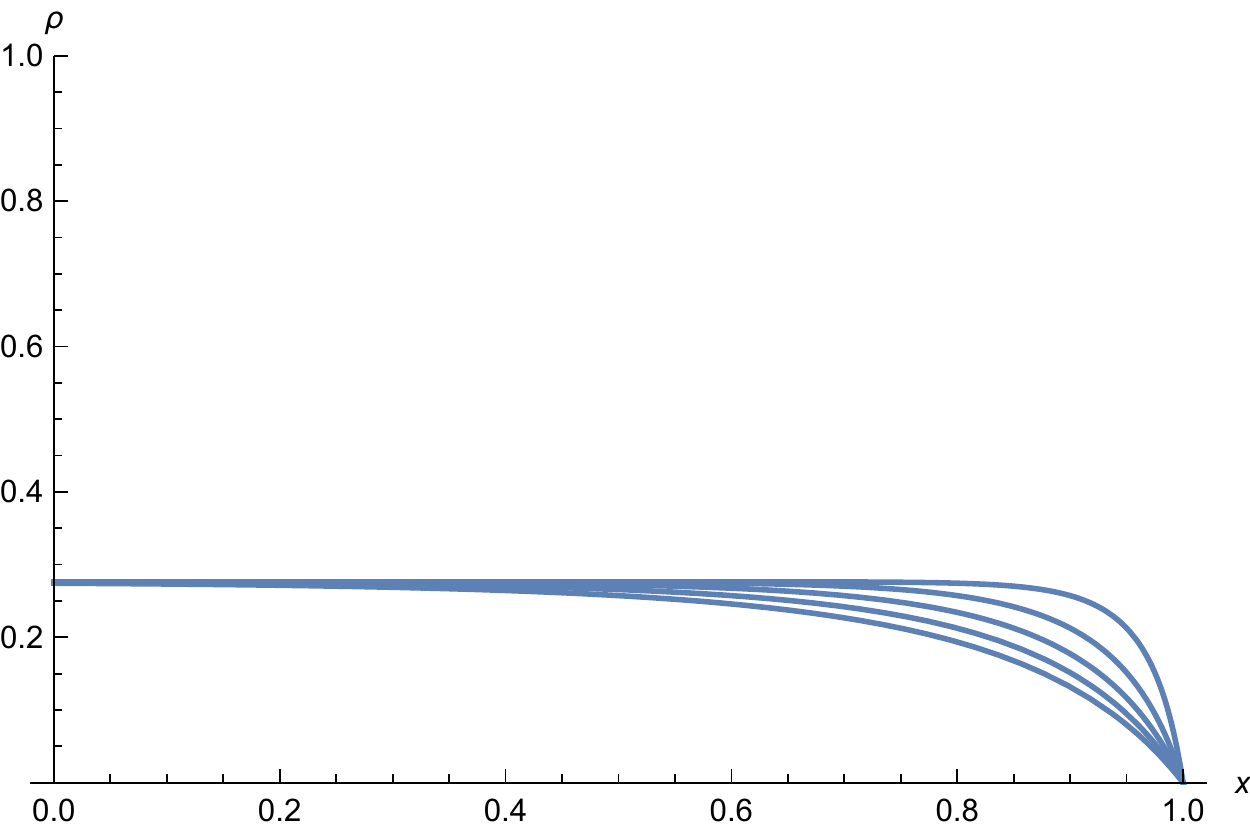}
\caption{\color{blue}Subcritical case.}
\label{fig:subcritical_case}
\end{subfigure}
\caption{{\color{blue}Dependence of congestion on viscosity.}}
\label{figure_dependence_congestion_on_viscosity}
\end{figure}   

For  {\color{blue} $j=0.2$}, the different solutions for the density $\rho$ show an upper bound when using different viscosities (from $0$ to $0.1$ with increments of $0.02$), see {\color{red}Figure~\ref{fig:subcritical_case}}.

To better understand the relation between viscosity and the critical current,
we solve the ODE \eqref{eq_ODE_stationary} with different viscosities and compute the critical current for which the density reaches one. 
At this density, the model breaks down as shown in {\color{red}Figure }\ref{figure_current_wrt_viscosity}.

\begin{figure}
\centering      
\includegraphics[width=0.3\textwidth]{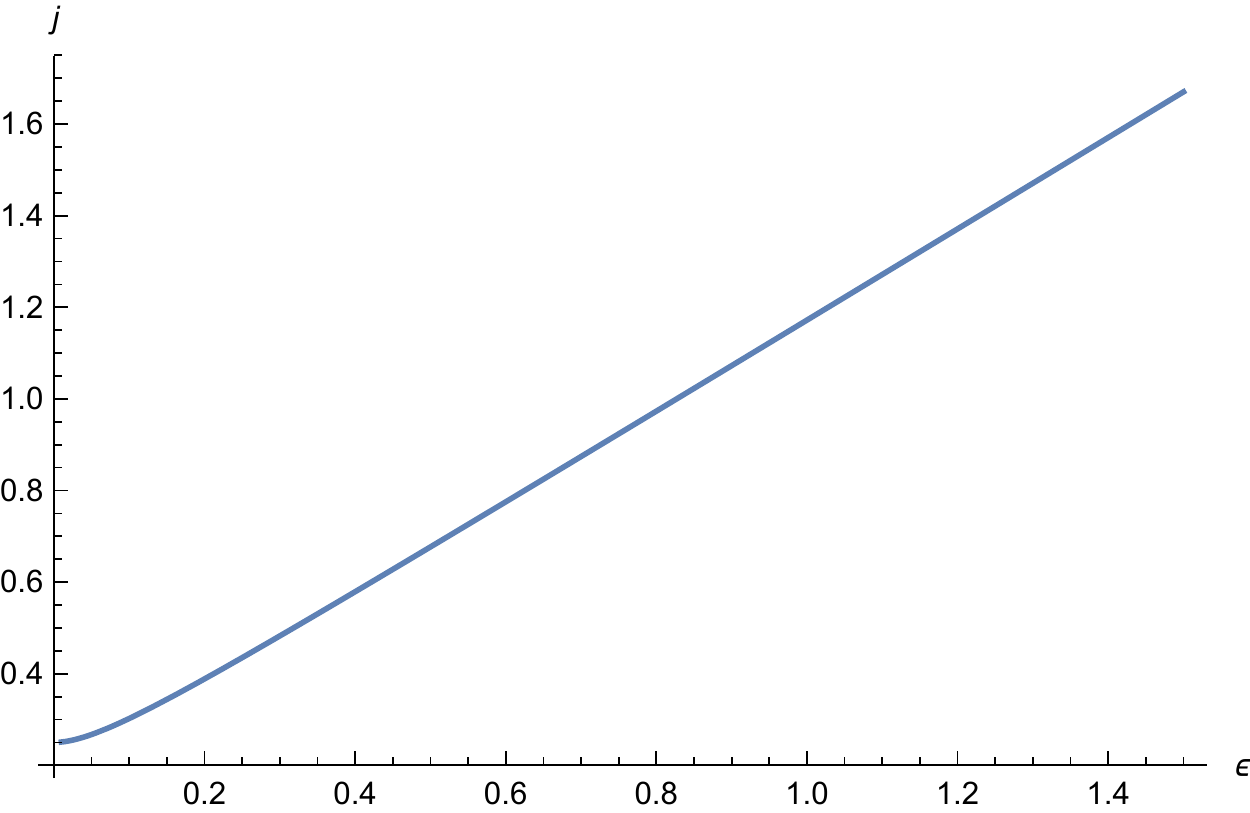}\\
\caption{{\color{blue}Critical current $j$ as a function of viscosity.}} 
\label{figure_current_wrt_viscosity}
\end{figure}

\section{Trend to equilibrium}
We investigate the relation between the solution to the stationary problem \eqref{eq_ODE_stationary} and  the time-dependent one \eqref{eq_pde_rho_1-dimensional}. The numerical solution to the stationary problem is calculated using an ODE solver. Now, the solution to the time-dependent case is computed using  the numerical approach we describe in Section \ref{section_numerical_methods}. As an example, we solve the problem with the following initial/boundary conditions, and $\varepsilon = 0.05$:
\begin{equation*}
\begin{cases}
\rho(0,t) = - 0.2 (1-e^{-10t}),\\
\rho(1,t) = 0,\\
\rho(x,0) = x^2 (1-x)^2.
\end{cases}
\end{equation*}
For large times, the stationary solution is an upper bound to the time-dependent one (the transient time where this behavior fails is due to the initial condition). 
This behavior is depicted in {\color{red}Figure~\ref{fig:dominance}}. The time-dependent solution is  the filled graph while the plot of the stationary one is depicted with wire mesh.

                \begin{figure}
                \centering
                \begin{subfigure}[b]{\sizefigure\textwidth}
                \includegraphics[width=\textwidth]{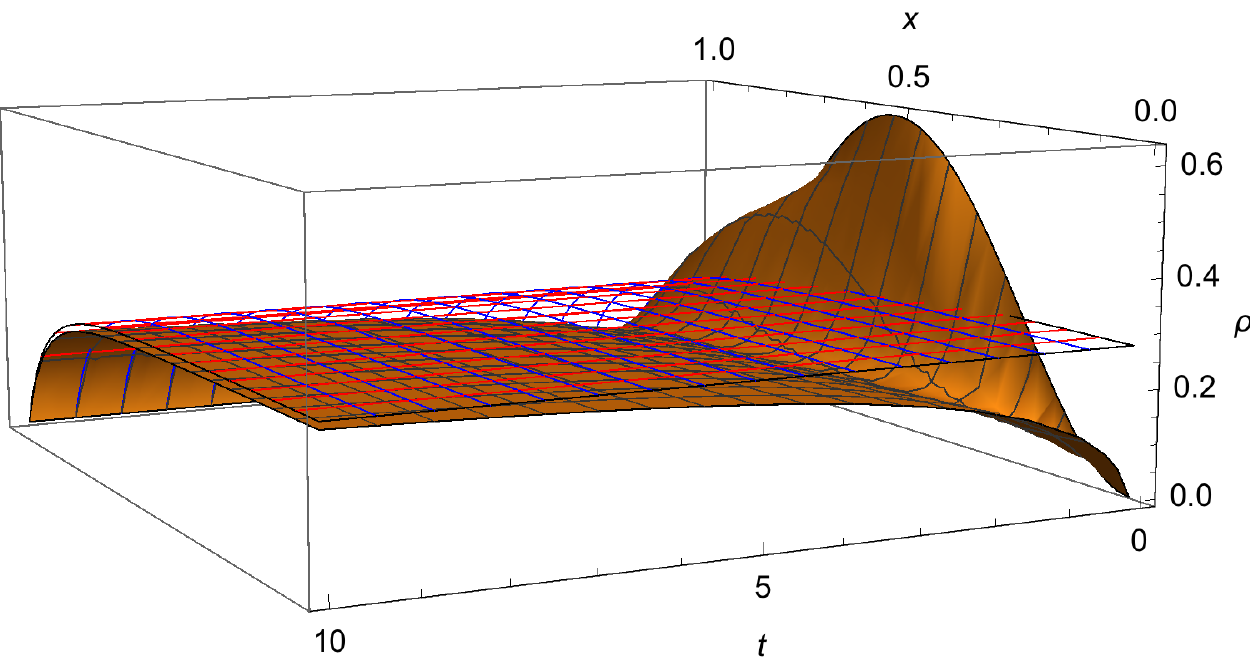}
                \caption{\color{blue}Dominance,}
                \label{fig:dominance}
                \end{subfigure}
                ~                                 \begin{subfigure}[b]{\sizefigure\textwidth}
                \includegraphics[width=\textwidth]{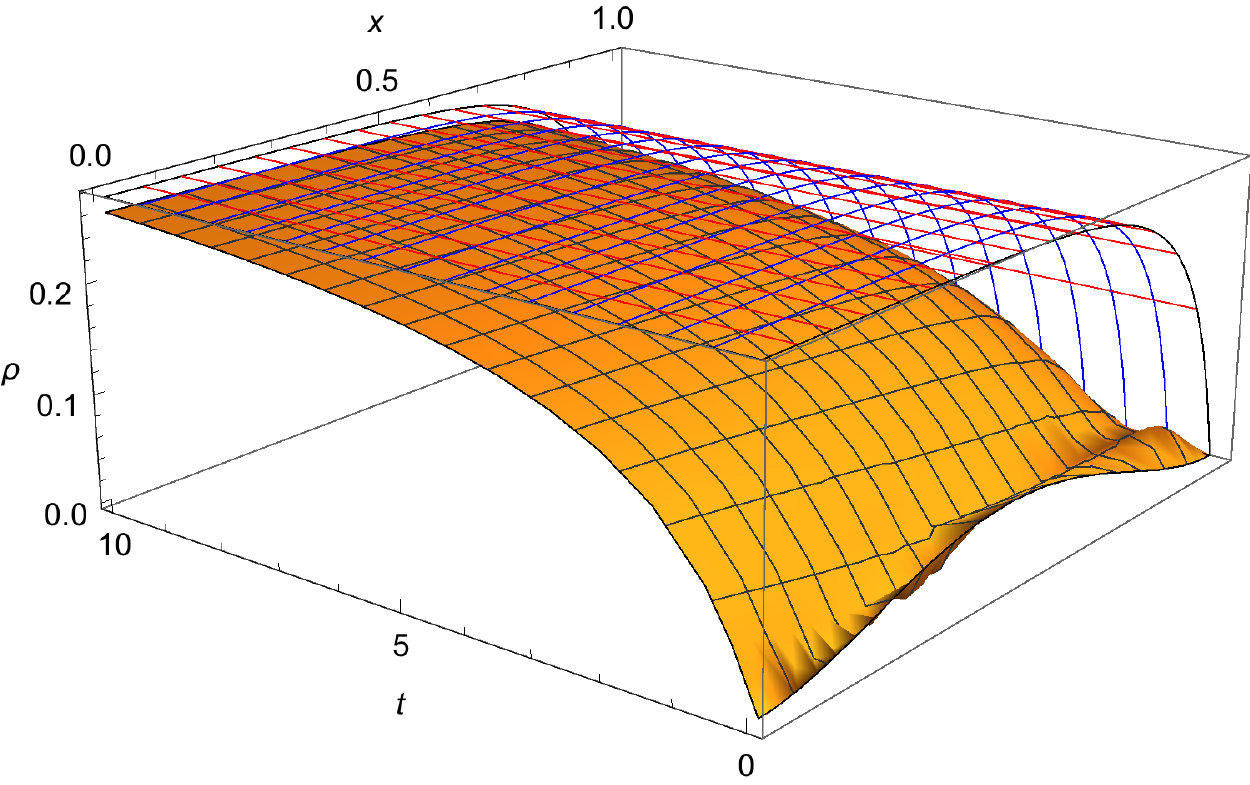}
                \caption{\color{blue}Trend to the equilibrium.}
                \label{fig:trend_equilibrium}
                \end{subfigure}
                \caption{{\color{blue}Stationary and time-dependent solutions.}}
                \label{}
                \end{figure}

The second aspect we observe is the trend to equilibrium of the time-dependent solution  in the subcritical case, {\color{red}Figure \ref{fig:trend_equilibrium}}. We are not aware of any proof or theoretical result on the asymptotic behavior of this problem. 

\section{Numerical approach}\label{section_numerical_methods}

Here, we describe a numerical approach to the Hughes model. Because the  Fokker-Planck equation is the adjoint of the linearization of a new Hamilton-Jacobi equation, we can use known methods for Hamilton-Jacobi equations to construct automatically schemes for the Fokker-Planck equation as we illustrate here.

To solve the Hughes model, we must use numerical methods that discretize $Du$\ in a consistent way for both equations. Our  approach is the following:  because the Fokker-Planck equation in \eqref{eq_hughes_fokker-planck} is the adjoint of the linearization of
the nonlinear Hamilton-Jacobi operator
\begin{equation}
\label{HJ2}
-u_t+(1-\rho)^2 \ \frac{|Du|^2}{2}-\varepsilon\Delta u,
\end{equation} 
we can treat both equations, this new Hamilton-Jacobi equation and the original Eikonal equation, via the same numerical method for Hamilton-Jacobi equations, for instance, using a monotone scheme.

In the examples discussed in the next section, we use a semi-discretization in space and treat the time variable as continuous (using a backward difference formula stiff solver for the time evolution). 

Let $h$ be the mesh size, $x_n$ the grid points and $u_n$, $\rho_n$ be the corresponding node values of $u$ and $\rho$. 
The non-linear operator in the Eikonal  equation in \eqref{eq_hughes_fokker-planck} at a point $x_n$
is discretized through the monotone scheme
\begin{align*}
N_n(u)\equiv & \frac{\max\{u_{n}-u_{n-1},0\}^2}{2h^2}+
\frac{\max\{u_{n}-u_{n+1},0\}^2}{2h^2}
-\frac{1}{(1-\rho_n)^2}.
\end{align*}
The operator \eqref{HJ2} is discretized as 
\begin{align*}
\tilde N_n(u)\equiv  (1-\rho_n)^2\left[\frac{\max\{u_{n}-u_{n-1},0\}^2}{2h^2}\right.
\left.+
\frac{\max\{u_{n}-u_{n+1},0\}^2}{2h^2}\right]
-\varepsilon \frac{u_{n+1}-2 u_n+u_{n-1}}{h^2}.
\end{align*}
The numerical scheme is given as
\begin{equation*}
\begin{cases}
N(u)=0,\\
\rho_t+(D_u \tilde N (u))^T \rho =0.
\end{cases}
\end{equation*}

Thanks to this adjoint structure, the second equation is discretized automatically using symbolic calculus. This approach is valid in arbitrary dimension and, by construction, has properties such as conservation of mass and positivity. 
\section{Numerical examples}\label{section_numerical_examples}
Now, we use the numerical approach from the previous section and present two examples. The first one deals with agents/people/pedestrians with the possibility of evacuating an one-dimensional domain from both sides. In the second example, we impose a current of agents entering the domain and we include reflecting boundary conditions on one of the sides of the one-dimensional domain.

\subsection{Example 1}
For the first numerical example we solve the Hughes model with a low viscosity $\varepsilon=0.01$, and the following initial/boundary conditions:\\
\begin{equation*}
\begin{cases}
\rho(0,t) = 0,\\
\rho(1,t) = 0,\\
\rho(x,0) = 0.9 \sin^2 (3 \pi x),\\
u(0,t)=0,\\
u(1,t)=0.
\end{cases}
\end{equation*}
These boundary conditions correspond to the exit problem. Agents have an initial distribution $\rho(x,0)$ and
seek to leave the interval $[0,1]$ by either $x=0$ or $x=1$. 
We plot the density~$\rho$ and the solution $u$ of the Eikonal, corresponding to the exit time, in {\color{red}Figure} \ref{numerical_example_1}.

                \begin{figure}
                \centering
                \begin{subfigure}[b]{\sizefigure\textwidth}
                \includegraphics[width=\textwidth]{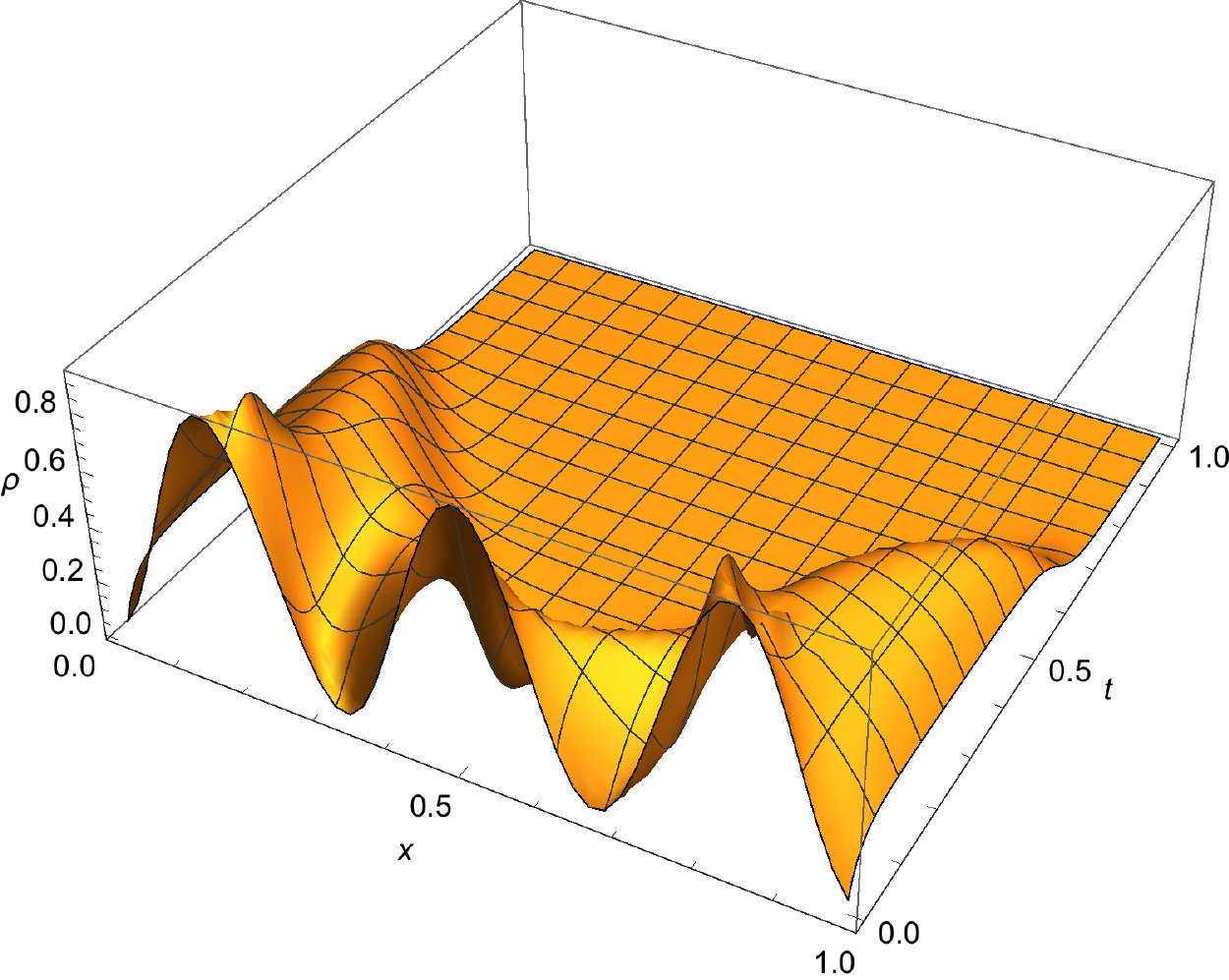}
                \caption{\color{blue}Density $\rho$,}
                \label{fig:numerical_example_1_density_rho}
                \end{subfigure}
                ~                                 \begin{subfigure}[b]{\sizefigure\textwidth}
                \includegraphics[width=\textwidth]{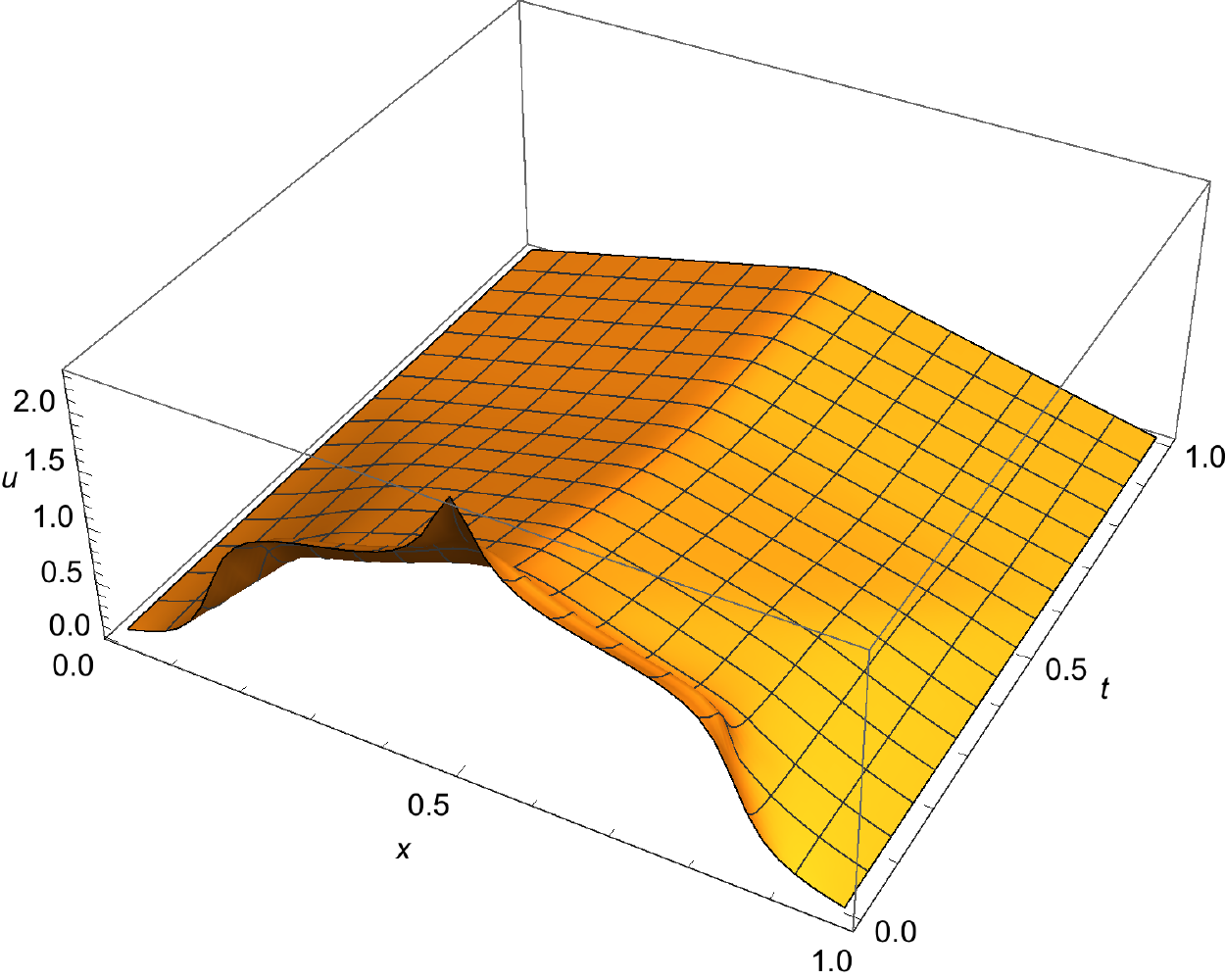}
                \caption{\color{blue}Eikonal solution $u$.}
                \label{fig:numerical_example_1_eikonal_solution}
                \end{subfigure}
                \caption{{\color{blue}Numerical example 1.}}
                \label{numerical_example_1}
                \end{figure}
                \subsection{Example 2}
        
        The second example is the  flow problem for
        \[
        \begin{cases}
        \rho(1,t)=0,\\
        \rho(x,0)=0.4 \sin^2 (3 \pi x),\\
        u(x,1)=0.
 
        \end{cases}
        \]
        In addition, at $x=0$, we impose for $\rho$ a flow-one condition:
        \[
        \rho(1-\rho)^2 Du +\varepsilon \rho_x=1,
        \]
        and, at $x~=~0,$ we impose reflecting boundary condition for $u$. Here,  we use a higher viscosity, $\varepsilon=0.1 $ and  present our results in {\color{red}Figure} \ref{numerical_example_2}.

                \begin{figure}
                \centering
                \begin{subfigure}[b]{\sizefigure\textwidth}
                \includegraphics[width=\textwidth]{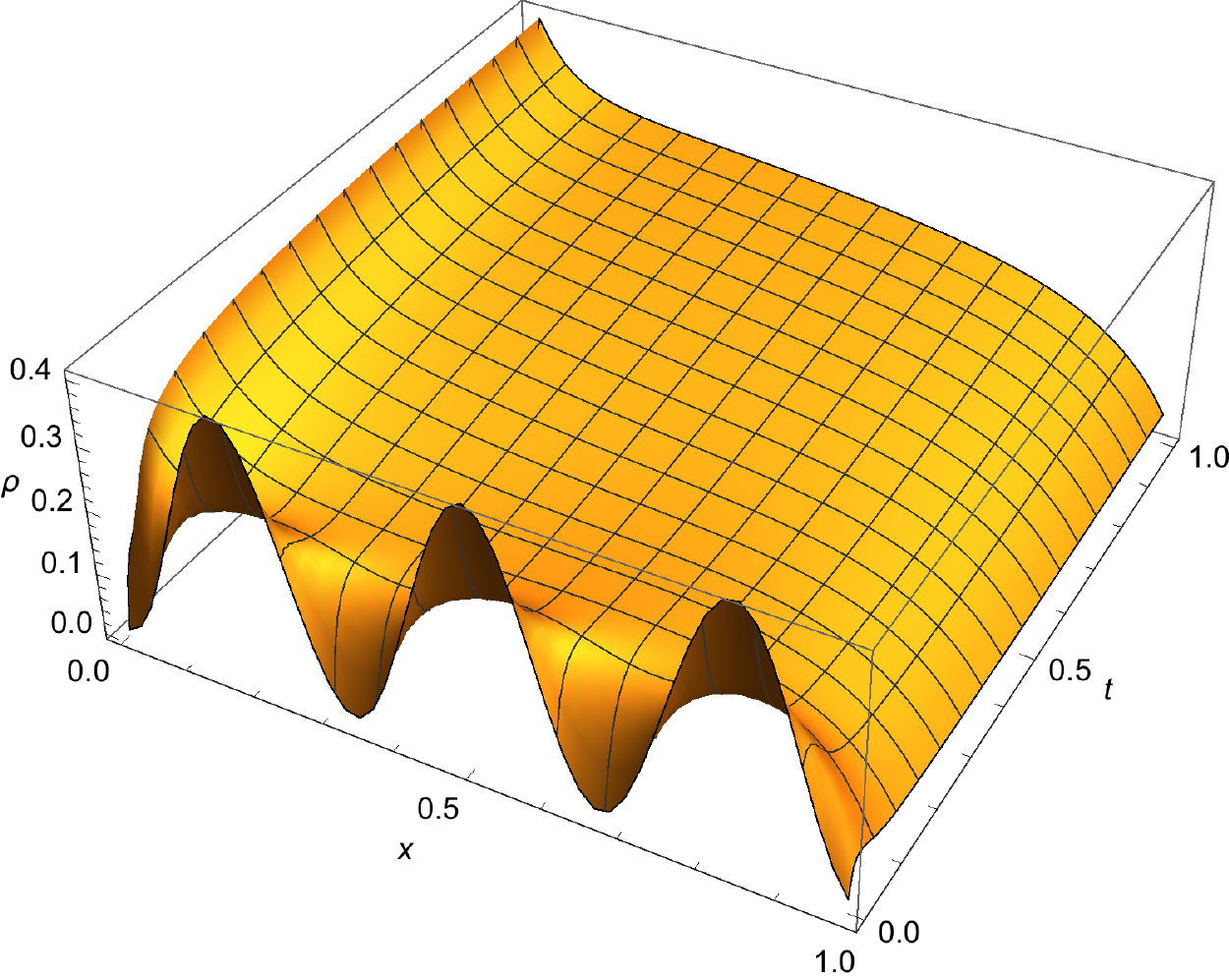}
                \caption{\color{blue}Density $\rho$}
                \label{fig:}
                \end{subfigure}
                ~                                 \begin{subfigure}[b]{\sizefigure\textwidth}
                \includegraphics[width=\textwidth]{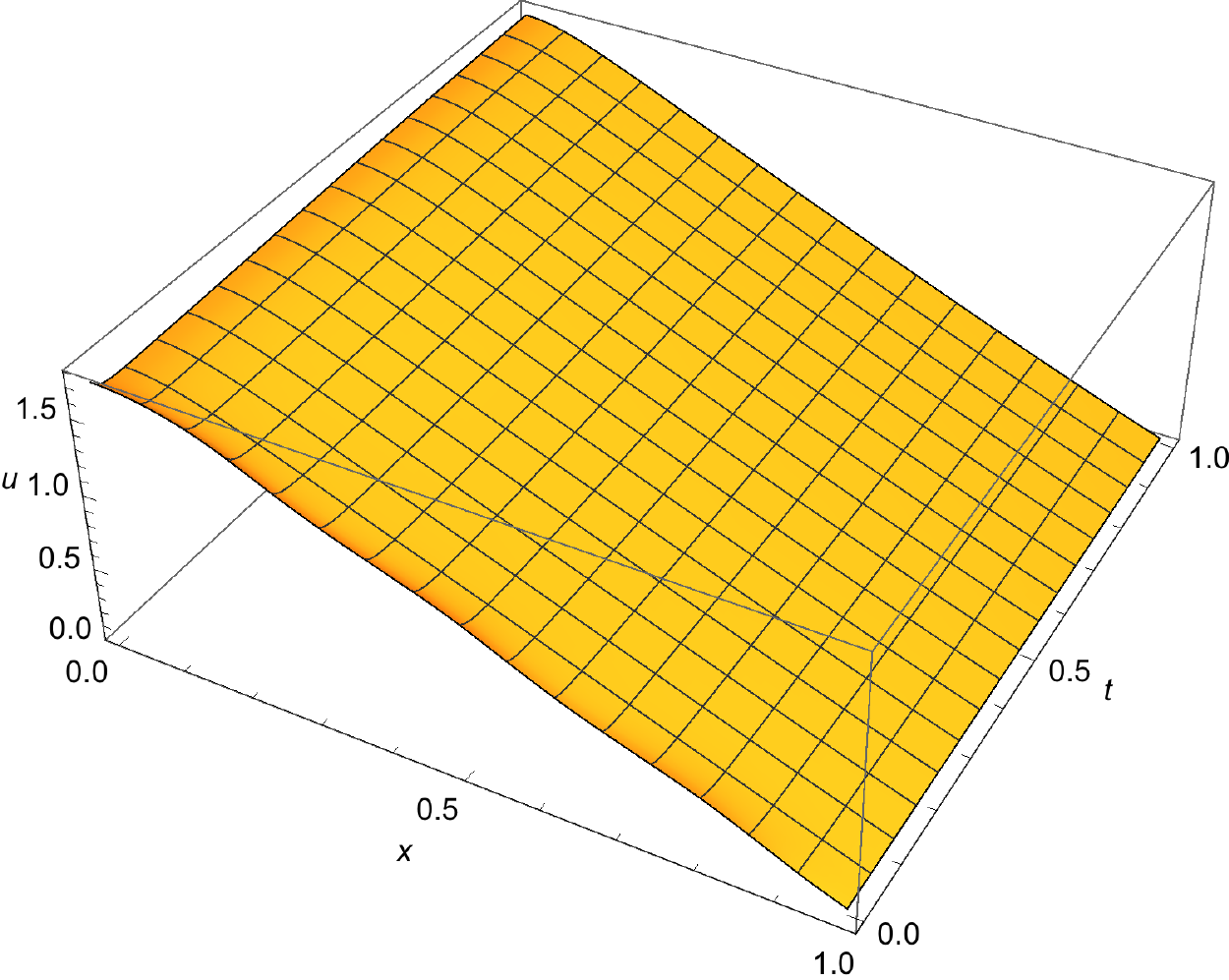}
                \caption{\color{blue}Eikonal solution $u$.}
                \label{fig:}
                \end{subfigure}
                \caption{{\color{blue} Numerical example 2.}}
                \label{numerical_example_2}
                \end{figure}

\section{Conclusions}
Here, we develop new a priori estimates for the Hughes model, which are an important step in understanding the wellposedness of the system for $\varepsilon>0$.

Next, we use radial solutions to prove the existence of shocks in dimension greater than one.
Consequently, the Hughes model without viscosity may fail to have smooth solutions.

Then, we uncover a new mechanism for the breakdown of classical solutions in the fixed current problem. Here, the critical density is reached without loss of regularity in~$\rho$.
Moreover, we examine the dependence of the critical current on the viscosity, and we present numerical evidence for the existence of a long-term limit. 

Finally, we describe a new method for the approximation of the Hughes model. Our method combines, in a novel way, monotone schemes for Hamilton-Jacobi equations with the adjoint structure of the Fokker-Planck equation, and applies to a wide range of related problems. 

\bibliographystyle{plain}
\bibliography{mfg}

\end{document}